\documentclass[english]{amsart}

\usepackage{esint}
\usepackage[svgnames]{xcolor} 
\usepackage[colorlinks,citecolor=red,pagebackref,hypertexnames=false,breaklinks]{hyperref}
\usepackage{pgf,tikz}
\usepackage{pdfsync}

\usepackage{dsfont}
\usepackage{url}
\usepackage[utf8]{inputenc}
\usepackage[T1]{fontenc}
\usepackage{lmodern}
\usepackage{babel}
\usepackage{mathtools}  
\usepackage{amssymb}
\usepackage{lipsum}
\usepackage{mathrsfs}
\usepackage{color}
\usepackage[skip=2.2pt plus 1pt, indent=12pt]{parskip}

\newtheorem{theorem}{Theorem}[section]

\newtheorem{lemma}{Lemma}[section]

\newtheorem{corollary}{Corollary}[section]

\numberwithin{equation}{section}

\title[Quantitative strong unique continuation]{Quantitative strong unique continuation for the Schr\"odinger operator with unbounded potential}

\author[Mourad Choulli]{Mourad Choulli}
\address{Universit\'e de Lorraine}
\email{mourad.choulli@univ-lorraine.fr}

\date{}
 

\begin{document}

\begin{abstract}
We revisit  \cite[Theorem 6.3]{JK}. By following the main ideas used to prove that theorem, we establish a quantitative version of strong unique continuation for the Schrödinger operator with an unbounded potential. We also show that combining this result with a global quantitative unique continuation from arbitrary interior data yields a global quantitative strong unique continuation.
\end{abstract}

\maketitle

\section{Introduction}

Let $n\ge 3$ be an integer, $p:=\frac{2n}{n+2}$ and $p':=\frac{p}{p-1}=\frac{2n}{n-2}$ its conjugate. The open ball centered at $0$ with radius $r>0$ will be denoted by $B_r$, and we set $B:=B_1$.

We say that $w\in L^2(B)$ vanishes of infinite order at $0$ if there exists $0<\gamma < 1$ such that for all integer $m>0$ we have
\[
r^{-m}\int_{B_r}|w|^2dx\le c_m,\quad 0< r<\gamma,
\] 
where $c_m>0$ is a constant.

Let $q\in L^{\frac{n}{2}}(B)$ and $P\in L^n(B,\mathbb{C}^n)$. We say that the Schr\"odinger operator $\Delta +q$ (resp. the differential inequality $|\Delta u|\le |qu|+|P\cdot \nabla u|$) admits the strong unique continuation property (hereinafter abbreviated as SUCP) in the class $W^{2,p}(B)$  if  any $u\in W^{2,p}(B)$  such that $(\Delta +q)u=0$  (resp. $|\Delta u|\le |qu|+|P\cdot \nabla u|$) in $B$ and $u$ vanishes of infinite order at $0$ must vanishes  in a neighborhood of $0$.

To our best knowledge, the first result establishing that the differential inequality $|\Delta u|\le |qu|$ has the SUCP in the class $W^{2,p}(B)$ with potential $q\in L^{\frac{n}{2}}(B)$ is due to Jerison and Kenig \cite{JK}. The SUCP for the differential inequality $|\Delta u|\le |P\cdot \nabla u|$ when $P\in L^{\frac{7n-2}{6}}(B,\mathbb{C}^n)$ was obtained by Regbaoui \cite{Re}. While Wolff \cite{Wo2} proved that the SUCP holds for the differential inequality $|\Delta u|\le |qu|+|P\cdot \nabla u|$ with $q\in L^{\frac{n}{2}}(B)$ and $P\in L^{\max\left(n,\frac{3n}{2}-2\right)}(B,\mathbb{C}^n)$. Sogge \cite{So} considered the case where $\Delta$ is replaced by a second order elliptic operator with smooth coefficients. He showed that the SUCP holds for the differential inequality  $|\Delta u|\le |qu|+|P\cdot \nabla u|$ with $q\in L^{\frac{n}{2}}(B)$ and $P\in L^\infty(B,\mathbb{C}^n) $. The result of Sogge \cite{So} was extended by Koch and Tataru \cite{KT} with the ``best possible conditions'', that is $q\in L^{\frac{n}{2}}(B)$ and $P\in L^n(B,\mathbb{R}^n)$ (in fact they considered an operator with an additional term of the form $\mathrm{div}(uR)$ with $R\in L^n(B,\mathbb{C}^n)$). By a counterexample, Jerison and Kenig showed in \cite[Remark 6.5]{JK} that $L^{\frac{n}{2}}(B)$ is the best possible space for the potential $q$. Reference \cite{KT} contains a brief history of the SUCP. Other references dealing with the SUCP for elliptic operators can found in \cite{KT,So}.

We aim in this article to quantify the SUCP for both the Schr\"odinger operator $\Delta +q$ and  the differential inequality $|\Delta u|\le |qu|+|P\cdot \nabla u|$. For sake of clarity, we limited ourselves to the case where the Schr\"odinger operator $\Delta +q$ acts on the unit ball $B$.  In a final section, we briefly outline the modifications required to obtain a variant of Theorem \ref{mt1.0}, as well as a quantitative SUCP for the differential inequality $|\Delta u|\le |qu|+|P\cdot \nabla u|$. We also present a quantitative SUCP and a global quantitative SUCP for the operator $\Delta +q$ acting on an arbitrary bounded domain in $\mathbb{R}^n$.

Before stating our main result precisely, we introduce some definitions and notations. As usual, the set of positive integers is denoted $\mathbb{N}$. For all  $m\in \mathbb{N}$, define $O_m$ as the set of functions $w\in L^2(B)$ satisfying
\[
\varrho_m (w):=\sup \left\{\eta^{-m}\|w\|_{L^2(B_\eta)};\; 0<\eta <  1\right\}<\infty.
\]
It is clear that $\varrho_m (w)<\infty$ if and only if there exists $0<\gamma <1$ so  that
\[
\varrho_m^\gamma (w):=\sup \left\{\eta^{-m}\|w\|_{L^2(B_\eta)};\; 0<\eta <  \gamma\right\}<\infty.
\]
Futhermore, $u$ vanishes of infinite order at $0$ if and only if $\displaystyle u\in \bigcap_{m\in \mathbb{N}}O_m$.

Let 
\[
\Lambda:=\left\{ \lambda >0;\; \mathrm{dist}\left(\lambda ,\mathbb{N}+\frac{n-2}{2}\right)= \frac{1}{2}\right\}
\]
and  $\dot{B}:=B\setminus\{0\}$. It follows from \cite[(3.4)]{JK} that
\begin{equation}\label{1}
\||x|^{-\lambda}u\|_{L^{p'}(B)}\le \kappa \||x|^{-\lambda}\Delta u\|_{L^p(B)},\quad u\in C_0^\infty(\dot{B}) ,\; \lambda \in \Lambda.
\end{equation}
Here and henceforth, $\kappa=\kappa(n)$ is a constant.

The closed ball of $L^{\frac{n}{2}}(B)$ centered at $0$ with radius $r>0$ will be denoted by $\mathbf{B}_r$.

To define the assumptions regarding the potential $q$, we need the following lemma, proved in Appendix \ref{appendixA}.

\begin{lemma}\label{lemma1}
There exists a constant $r_\ast=r_\ast(n)>0$ with the property that for all  $0<r_0<r_\ast$, there exists $\kappa_0=\kappa_0(n,r_0)>0$ such that    we have
\begin{equation}\label{4}
\||x|^{-\lambda}u\|_{L^{p'}(B)}\le \kappa_0 \||x|^{-\lambda}(\Delta +q)u\|_{L^p(B)},\quad u\in C_0^\infty(\dot{B}),\; q\in\mathbf{B}_{r_0},\; \lambda \in \Lambda.
\end{equation}
\end{lemma}

For convenience, recall that $W^{2,p}(B)$ embeds continuously into $H^1(B)$ and that $H^1(B)$ embeds continuously into $L^{p'}(B)$. Let 
\[
\sigma:=\sup \left\{ \|w\|_{L^{p'}(B)};\; w\in H_0^1(B),\;  \|\nabla w\|_{L^2(B)}=1\right\},
\]
We fix $0<r_0<r_\ast$, where $r_\ast$ is as in Lemma \ref{lemma1}, and $\vartheta>0$ satisfying $2\vartheta\sigma^2<1$. Then, let $\varsigma:=\inf(r_0,\vartheta)$, and $Q:=\mathbf{B}_\varsigma$.

Throughout this text,  we arbitrarily fix $0<\rho <\frac{1}{2}$ and $\phi\in C_0^\infty (B)$ such that $\phi =1$ in $B_{\frac{1}{2}}$ and $0\le \phi \le 1$. Our main result is the following theorem

\begin{theorem}\label{mt1.0}
There exists a constant $\mathbf{c}=\mathbf{c}(n,r_0,\sigma,\vartheta)>0$ such that for all $m\in \mathbb{N}$, $\lambda \in \Lambda$, $q\in Q$, $u\in W^{2,p}(B)\cap O_m$ and $0<\eta <  \frac{\rho}{4}$ we have 
\begin{align}
&\mathbf{c}\|u\|_{L^2(B_{\frac{\rho}{2}})}\le   \eta^{-\lambda}\|(\Delta+q) u\|_{L^p(B)}\label{mat1.0}
\\
&\hskip 3cm + 2^{-(\lambda+1)}\|(\Delta+q) (\phi u)\|_{L^{p}(B)}+3^m\eta^{m-(\lambda+3)}\varrho_m(u).\nonumber
\end{align}
\end{theorem}

Let's explain why Theorem \ref{mt1.0} quantifies the SUCP for the operator $\Delta +q$, $q\in Q$, in the class $W^{2,p}(B)$. Let $u\in W^{2,p}(B)$ such that $(\Delta+q)u=0$ in $B$ and $u$ vanishes of infinite order at $0$. For all $\lambda \in \Lambda$, $0<\eta <  \frac{\rho}{4}$ and $m=[\lambda]+4$, where $[\lambda]$ is the entire part of $\lambda$, \eqref{mat1.0} yields
\[
\mathbf{c}\|u\|_{L^2(B_{\frac{\rho}{2}})}\le 2^{-(\lambda+1)}\|(\Delta+q) (\phi u)\|_{L^{p}(B)}+3^{[\lambda]+4}\eta^{[\lambda]+1-\lambda}\varrho_{[\lambda]+4}(u).
\]
Taking the limit when $\eta$ tends to zero in this inequality, we find
\[
\mathbf{c}\|u\|_{L^2\left(B_{\rho/2}\right)}\le 2^{-(\lambda+1)}\|(\Delta+q) (\phi u)\|_{L^{p}(B)}.
\]
Let $(\lambda_j)_{j\in \mathbb{N}}$ be a sequence of $\Lambda$ tending to $\infty$ when $j$ goes to $\infty$. From the preceding inequality, we have
\[
\mathbf{c}\|u\|_{L^2\left(B_{\rho/2}\right)}\le 2^{-(\lambda_j+1)}\|(\Delta+q) (\phi u)\|_{L^{p}(B)},\quad j\in \mathbb{N},
\]
from which we obtain that $u=0$ in $B_{\frac{\rho}{2}}$. Hence, $u=0$ in $\bigcup_{0<\rho <\frac{1}{2}}B_{\frac{\rho}{2}}=B_{\frac{1}{4}}$. In other words, we proved the following result.

\begin{corollary}\label{corollary1}
For $q\in Q$, let $u\in W^{2,p}(B)$  such that  $(\Delta +q)u=0$ in $B$ and $u$ vanishes of infinite order at $0$. Then $u=0$ in $B_{\frac{1}{4}}$.
\end{corollary}

The key ingredient in the proof of Theorem \ref{mt1.0} is the Carleman inequality \eqref{1} due to Jerison and Kenig \cite{JK}. For elliptic operators with variable leading coefficients, a Carleman inequality is indeed available \cite{KT}. However, this Carleman estimate is less flexible to use than \eqref{1}, and we do not believe it is the appropriate tool for establishing a SUCP for operators with variable leading coefficients. We note the paper by Malinnikova and Vessella \cite{MV}, in which the authors-relying on the Carleman inequality of Koch and Tataru \cite{KT}-proved a quantitative strong unique continuation result from a set of positive measure.

Under the additional assumption $q\in L^n(B)$, we combine \eqref{mat1.0} with \cite[Theorem 1.2 and comments in Section 1.3]{Ch} to obtain a global quantitative SUCP. Precisely, we have the following result, where
\[
\mathbf{Q}:=\left\{q\in L^{n}(B);\; \|q\|_{L^n(B)}\le |B|^{-\frac{1}{n}}\varsigma \right\}\; (\subset Q).
\]

\begin{theorem}\label{mt2}
Let $0<s,t<\frac{1}{2}$. For all $0<r <1$, $0<\eta<\frac{\rho}{4}$, $m\in \mathbb{N}$, $\lambda \in \Lambda$, $q\in \mathbf{Q}$ and  $u\in H^2(B)\cap O_m$ we have
\begin{align*}
&\mathbf{c}\|u\|_{H^{\frac{3}{2}+t}(B)}\le r^{\zeta}\|u\|_{H^2(B)} 
\\
&\hskip 1.5cm +e^{\mathfrak{b}r^{-\nu}}\Big(\eta^{-\lambda}\|(\Delta+q) u\|_{L^2(B)} 
\\
&\hskip 4cm + 2^{-(\lambda+1)}\|(\Delta+q) (\phi u)\|_{L^{p}(B)}+3^m\eta^{m-(\lambda+3)}\varrho_m(u)\Big),
\end{align*}
where $\mathbf{c}=\mathbf{c}(n,s, t, r_0,\sigma,\vartheta)>0$, $\mathfrak{b}=\mathfrak{b}(n, s, t, r_0,\sigma,\vartheta)>0$, $\nu=\nu(n)>0$ are constants and $\zeta =\min \left(\frac{\frac{1}{2}-t}{1+t},\frac{s}{4}\right)$.
\end{theorem}

Proceeding as for Corollary \ref{corollary1}, we deduce  the following global SUCP from Theorem \ref{mt2}.

\begin{corollary}\label{corollary2}
Let $q\in \mathbf{Q}$. If  $u\in H^2(B)$ is such that $(\Delta +q) u=0$ in $B$ and $u$ vanishes of infinite order at $0$, then $u=0$ in $B$.
\end{corollary}

Let $E\subset B$ be a set of positive measure. Proceeding as in \cite[Section 3]{Ch}, we obtain the following consequence of Corollary \ref{corollary2}.

\begin{corollary}\label{corollary2.0}
Let $q\in \mathbf{Q}$ and $u\in H^2(B)$ such that $(\Delta +q) u=0$ in $B$ and $u=0$ in $E$. Then, $u=0$ in $B$.
\end{corollary}

\section{Proof of Theorem \ref{mt1.0}}

\subsection{Caccioppoli's  inequality}

\begin{lemma}\label{lemma0}
Let $0<r,h <1$ such that $0<r+h <1$. There exists a constant $\mathbf{c}=\mathbf{c}(n, \sigma, \vartheta)>0$ with the property that for all $u\in W^{2,p}(B)$ and $q\in Q$ we have
\begin{equation}\label{ca1}
\mathbf{c}\|\nabla u\|_{L^2(B_r)}\le \|(\Delta+q) u\|_{L^p(B)}+h^{-1}\|u\|_{L^2(B_{r+h})}.
\end{equation}
\end{lemma}

\begin{proof}
Let $\chi \in C_0^\infty (B_{r+h})$ such that  $0\le \chi \le 1$, $\chi =1$ in a neighborhood of $B_r$ and $|\partial^\alpha\chi |\le c_0h^{-1}$ for all $|\alpha|=1$, where $c_0>0$ is a universal constant.

Let $u\in W^{2,p}(B)$ and $\epsilon>0$. Applying H\"older's inequality, we obtain
\[
\left| \int_B\chi^2 u(\Delta+q) udx\right| \le \|(\Delta+q) u\|_{L^p(B)}\|\chi u\|_{L^{p'}(B)}
\]
and hence
\begin{align*}
\left| \int_B\chi^2 u(\Delta+q) udx\right| &\le (2\epsilon)^{-1}\|(\Delta+q) u\|_{L^p(B)}^2+2^{-1}\epsilon\|\chi u\|_{L^{p'}(B)}^2
\\
&\le (2\epsilon)^{-1}\|(\Delta+q) u\|_{L^p(B)}^2+2^{-1}\epsilon\sigma^2\|\nabla (\chi u)\|_{L^2(B)}^2.
\end{align*}
That is we have
\begin{align}
\left| \int_B\chi^2 u(\Delta+q) udx\right| \le (2\epsilon)^{-1}&\|(\Delta+q) u\|_{L^p(B)}^2\label{ca2}
\\
&+\epsilon\sigma^2\left(\|u\nabla \chi\|_{L^2(B)}^2+\|\chi \nabla u\|_{L^2(B)}^2\right).\nonumber
\end{align}

On the other hand, performing an  integration by parts, we get
\[
\int_B\chi^2 u\Delta udx=-\int_B\chi^2|\nabla u|^2dx-\int_B u\nabla \chi^2 \cdot \nabla udx.
\]
Thus, 
\[
\|\chi \nabla u\|_{L^2(B)}^2\le \left|\int_B\chi^2 u\Delta udx\right|+2\left|\int_B u\chi \nabla \chi \cdot \nabla udx\right|.
\]
By applying Cauchy-Schwarz's inequality and then a convexity inequality to the second term of the right hand side of inequality above, we get
\[
\|\chi \nabla u\|_{L^2(B)}^2\le \left|\int_B\chi^2 u\Delta udx\right|+\epsilon^{-1}\|u\nabla \chi\|_{L^2(B)}^2+\epsilon\|\chi \nabla u\|_{L^2(B)}^2.
\]
Hence,
\[
(1-\epsilon)\|\chi \nabla u\|_{L^2(B)}^2\le \left|\int_B\chi^2 u\Delta udx\right|+\epsilon^{-1}\|u\nabla \chi\|_{L^2(B)}^2,
\]
from which we obtain
\begin{equation}\label{ca3}
(1-\epsilon)\|\chi \nabla u\|_{L^2(B)}^2\le \left|\int_B\chi^2 u(\Delta+q) udx\right|+ \left|\int_Bq\chi^2 u^2dx\right|+\epsilon^{-1}\|u\nabla \chi\|_{L^2(B)}^2.
\end{equation}
By applying twice H\"older's inequality, we obtain
\[
\left|\int_Bq\chi^2 u^2dx\right|\le \|q\chi u\|_{L^p(B)}\|\chi u\|_{L^{p'}(B)}\le \|q\|_{L^{\frac{n}{2}}(B)}\|\chi u\|_{L^{p'}(B)}^2
\]
Hence, we have
\begin{align*}
\left|\int_Bq\chi^2 u^2dx\right|&\le  \|q\|_{L^{\frac{n}{2}}(B)}\sigma^2\|\nabla (\chi u)\|_{L^2(B)}^2
\\
& \le  2\|q\|_{L^{n/2}(B)}\sigma^2\left(\|u\nabla \chi \|_{L^2(B)}^2+\|\chi \nabla u\|_{L^2(B)}^2\right)
\\
& \le  2\vartheta\sigma^2\left(\|u\nabla \chi \|_{L^2(B)}^2+\|\chi \nabla u\|_{L^2(B)}^2\right).
\end{align*}
This inequality in \eqref{ca3} gives
\begin{equation}\label{ca4}
(1-\epsilon-2\vartheta\sigma^2)\|\chi \nabla u\|_{L^2(B)}^2\le \left|\int_B\chi^2 u(\Delta+q) udx\right|+(\epsilon^{-1}+2\vartheta\sigma^2)\|u\nabla \chi\|_{L^2(B)}^2.
\end{equation}
Putting \eqref{ca2} in \eqref{ca4}, we end up getting
\begin{align*}
(1-2\vartheta\sigma^2-\epsilon(1+\sigma^2))\|\chi \nabla u\|_{L^2(B)}^2 &\le (2\epsilon)^{-1}\|(\Delta+q) u\|_{L^p(B)}^2
\\
&\qquad +(\epsilon^{-1}+\epsilon\sigma^2+2\vartheta\sigma^2)\|u\nabla \chi\|_{L^2(B)}^2.
\end{align*}
The expected inequality follows by taking $\epsilon=\frac{1-2\vartheta\sigma^2}{2+2\sigma^2}$.
\end{proof}

\subsection{Proof of Theorem \ref{mt1.0}}

Let  $(\psi_\eta)_{0<\eta < \frac{\rho}{4}}$ be a family of $C^\infty(\mathbb{R}^n)$ such that $0\le \psi_\eta\le 1$, $\psi_\eta(x)=0$ if $|x|\le \eta$, $\psi_\eta(x)=1$ if $|x|\ge 2\eta$ and
\[
|\partial^\alpha \psi_\eta|\le c\eta^{-|\alpha|},\quad |\alpha|\le 2,\; 0<\eta <\frac{\rho}{4},
\]
where $c>0$ is a universal constant. 

Let $m\in \mathbb{N}$, $q\in Q$, $u\in W^{2,p}(B)\cap O_m$ and $v:=\phi u$. Let $0<\eta< \frac{\rho}{4}$ be arbitrarily fixed. By density, we get from \eqref{4} the following inequality
\begin{equation}\label{5}
\||x|^{-\lambda}\psi_\eta v\|_{L^{p'}(B)}\le \kappa_0 \||x|^{-\lambda}(\Delta+q) (\psi_\eta v)\|_{L^{p}(B)},
\end{equation}
where $\kappa_0$ is the constant  in \eqref{4}. Therefore, we have
\begin{align*}
\||x|^{-\lambda}\psi_\eta u\|_{L^{p'}(B_{\frac{\rho}{2}})}\le \kappa_0 \||x|^{-\lambda}&(\Delta+q) (\psi_\eta v)\|_{L^{p}(B_\rho)}
\\
&+\kappa_0\||x|^{-\lambda}(\Delta+q) (\psi_\eta v)\|_{L^{p}(\{\rho<|x|<1\})},
\end{align*}
and hence
\begin{align}
\||x|^{-\lambda}\psi_\eta u\|_{L^{p'}(B_{\frac{\rho}{2}})}\le \kappa_0 \||x|^{-\lambda}(\Delta+q) &(\psi_\eta u)\|_{L^{p}(B_\rho)}\label{5.0}
\\
&+\kappa_0 \rho^{-\lambda}\|(\Delta+q) v\|_{L^{p}(B)}.\nonumber
\end{align}
In consequence, we have
\begin{align}
\||x|^{-\lambda} u\|_{L^{p'}(\{2\eta <|x]<\frac{\rho}{2}\})}\le \kappa_0 \||x|^{-\lambda}(\Delta+q) &(\psi_\eta u)\|_{L^{p}(B_{\rho})}\label{5.1}
\\
&+\kappa_0 \rho^{-\lambda}\|(\Delta+q) v\|_{L^{p}(B)}.\nonumber
\end{align}
On the other hand,  applying H\"older's inequality, we obtain
\[
\|u\|_{L^2(\{2\eta <|x]<\frac{\rho}{2}\})}\le \||x|^\lambda\|_{L^n(\{2\eta <|x|<\frac{\rho}{2}\})}\||x|^{-\lambda} u\|_{L^{p'}(\{2\eta <|x]<\frac{\rho}{2}\})}.
\]
Since
\[
\||x|^\lambda\|_{L^n(\{2\eta <|x]<\frac{\rho}{2}\})}\le \||x|^\lambda\|_{L^n(B_{\frac{\rho}{2}})}=\left[\frac{\omega_n}{n(\lambda+1)}\right]^{1/n}\left(\frac{\rho}{2}\right)^{\lambda+1},
\]
where $\omega_n=|\mathbb{S}^{n-1}|$, we get
\begin{equation}\label{5.2}
\|u\|_{L^2(\{2\eta <|x]<\frac{\rho}{2}\})}\le\left[\frac{\omega_n}{n(\lambda+1)}\right]^{\frac{1}{n}}\left(\frac{\rho}{2}\right)^{\lambda+1}\||x|^{-\lambda} u\|_{L^{p'}(\{2\eta <|x]<\frac{\rho}{2}\})}.
\end{equation}
By putting together \eqref{5.1} and \eqref{5.2}, we obtain
\begin{align*}
&\|u\|_{L^2(B_{\frac{\rho}{2}})}\le \kappa_1 (1+\lambda)^{-\frac{1}{n}}\left(\frac{\rho}{2}\right)^{\lambda+1} \||x|^{-\lambda}(\Delta+q)(\psi_\eta u)\|_{L^{p}(B_\rho)}
\\
&\hskip 5cm +\kappa_1 \rho2^{-(\lambda+1)}\|(\Delta+q) v\|_{L^{p}(B)}+\|u\|_{L^2(B_{2\eta})},
\end{align*}
where $\kappa_1=\kappa_0\left(\frac{\omega_n}{n}\right)^{\frac{1}{n}}$. Hence, the following inequality holds
\begin{align}
&\|u\|_{L^2(B_{\frac{\rho}{2}})}\le \kappa_1  \||x|^{-\lambda}(\Delta+q)(\psi_\eta u)\|_{L^{p}(B_\rho)}\label{5.3}
\\
&\hskip 3cm +\kappa_1 2^{-(\lambda+1)}\|(\Delta+q) v\|_{L^{p}(B)}+\|u\|_{L^2(B_{2\eta})}.\nonumber
\end{align}

Next, as 
\[
\Delta (\psi_\eta u)=\psi_\eta \Delta u+ 2\nabla \psi_\eta \cdot \nabla u+\Delta \psi_\eta u
\]
and  $\mbox{supp}(2\nabla \psi_\eta\cdot \nabla u+\Delta \psi_\eta u)\subset \{\eta<|x|<2\eta\}$, we get
\begin{align}
\||x|^{-\lambda}(\Delta+q) (\psi_\eta u)\|_{L^{p}(B_\rho)}\le  &\eta^{-\lambda}\|(\Delta+q) u\|_{L^p(B)} \label{7}
\\
&+ c_0\eta ^{-(\lambda+2)}\left(\|\nabla u\|_{L^2(B_{2\eta})}+\|u\|_{L^2(B_{2\eta})}\right),\nonumber
\end{align}
for some constant $c_0=c_0(n)>0$.

In the remaining part of this proof, $\mathbf{c}=\mathbf{c}(n,r_0,\vartheta)>0$ will denote a generic constant.

From inequality \eqref{ca1}, we obtain
\[
\mathbf{c}\|\nabla u\|_{L^2(B_{2\eta})}\le \|(\Delta+q) u\|_{L^p(B)}+ \eta^{-1}\|u\|_{L^2(B_{3\eta})}.
\]

This inequality in \eqref{7} yields
\begin{align}
\mathbf{c}\||x|^{-\lambda}(\Delta+q) (\psi_\eta u)&\|_{L^{p}(B_\rho)} \label{8}
\\
&\le \eta^{-\lambda}\|(\Delta+q) u\|_{L^p(B)}+\eta^{-(\lambda+3)}\|u\|_{L^2(B_{3\eta})}.\nonumber
\end{align}
Therefore, it follows from \eqref{5.3} and \eqref{8}
\begin{align*}
&\mathbf{c}\|u\|_{L^2(B_{\frac{\rho}{2}})}\le   \eta^{-\lambda}\|(\Delta+q) u\|_{L^p(B)}+ 2^{-(\lambda+1)}\|(\Delta+q) v\|_{L^{p}(B)}
\\
&\hskip 8cm +\eta^{-(\lambda+3)}\|u\|_{L^2(B_{3\eta})}.
\end{align*}
We end up getting
\begin{align*}
&\mathbf{c}\|u\|_{L^2(B_{\frac{\rho}{2}})}\le   \eta^{-\lambda}\|(\Delta+q) u\|_{L^p(B)}+ 2^{-(\lambda+1)}\|(\Delta+q) (\phi u)\|_{L^{p}(B)}
\\
&\hskip 8.5cm +3^m\eta^{m-(\lambda+3)}\varrho_m(u).
\end{align*}
This is the expected inequality.

\section{Other results}

\subsection{Quantitative SUCP around a fixed potential}

Let $q_0\in L^{\frac{n}{2}}(B)$ and $0<\underline{\rho}<\frac{1}{2}$ such that $\sigma^2\|q_0\|_{L^{\frac{n}{2}}(B_{\underline{\rho}})}<\frac{1}{4}$. Let 
\[
\underline{Q}:=\left\{q\in L^{\frac{n}{2}}(B);\; \sigma^2\|q-q_0\|_{L^{\frac{n}{2}}(B_{\underline{\rho}})}<\frac{1}{4}\right\}.
\]
The following inequality will be useful later on.
\begin{equation}\label{c1}
\sigma^2\|q\|_{L^{\frac{n}{2}}(B_{\underline{\rho}})}<\frac{1}{2},\quad q\in \underline{Q}.
\end{equation}

Under the definitions and the notations of the proof of Theorem \ref{mt1.0} with $\rho=\underline{\rho}$, instead of \eqref{5.0} we have
\[
\||x|^{-\lambda}\psi_\eta u\|_{L^{p'}(B_{\underline{\rho}})}\le \kappa_0 \||x|^{-\lambda}\Delta (\psi_\eta u)\|_{L^{p}(B_{\underline{\rho}})}
+\kappa_0 \rho^{-\lambda}\|\Delta v\|_{L^{p}(B)}.
\]
In light of \eqref{c1}, this inequality implies for all $q\in \underline{Q}$
\[
\||x|^{-\lambda}\psi_\eta u\|_{L^{p'}(B_{\underline{\rho}})}\le 2\kappa_0 \||x|^{-\lambda}(\Delta+q) (\psi_\eta u)\|_{L^{p}(B_{\underline{\rho}})}
+2\kappa_0 \rho^{-\lambda}\|\Delta v\|_{L^{p}(B)},
\]
In particular, we have for all $q\in \underline{Q}$
\[
\||x|^{-\lambda}\psi_\eta u\|_{L^{p'}(B_{\frac{\underline{\rho}}{2}})}\le 2\kappa_0 \||x|^{-\lambda}(\Delta+q) (\psi_\eta u)\|_{L^{p}(B_{\underline{\rho}})}
+2\kappa_0 \rho^{-\lambda}\|\Delta v\|_{L^{p}(B)},
\]
As the Caccioppoli's inequality \eqref{ca1} remains valid when $q\in Q$ is replaced by $q\in \underline{Q}$ and $0<r+h<1$ is replaced by $0<r+h<\underline{\rho}$, by modifying slightly the proof of Theorem \ref{mt1.0}, we get the following result.

\begin{theorem}\label{mt1.1}
There exists a constant $\mathfrak{c}=\mathfrak{c}(n)>0$ such that for all $m\in \mathbb{N}$, $\lambda \in \Lambda$, $q\in \underline{Q}$, $u\in W^{2,p}(B)\cap O_m$ and $0<\eta <  \frac{\underline{\rho}}{4}$ we have 
\begin{align*}
&\mathfrak{c}\|u\|_{L^2(B_{\frac{\underline{\rho}}{2}})}\le   \eta^{-\lambda}\|(\Delta+q) u\|_{L^p(B)}
\\
&\hskip 3cm+ 2^{-(\lambda+1)}\|\Delta (\phi u)\|_{L^{p}(B)}
 +3^m\eta^{m-(\lambda+3)}\varrho_m(u).
\end{align*}
\end{theorem}

An immediate consequence of this theorem is

\begin{corollary}\label{corollaryc1}
Let $q\in L^{\frac{n}{2}}(B)$. If $u\in W^{2,p}(B)$ is such that $(\Delta +q)u=0$ in $B$ and $u$ vanishes of infinite order at $0$, then $u=0$ in a neighborhood $V$ of $0$.
\end{corollary}

Of course, $V$ in Corollary \ref{corollaryc1} depends on $q$.

\subsection{Quantitative SUCP for differential inequalities}

Here again, the notations are those of the proof of Theorem \ref{mt1.0}.   As at the beginning of the proof of Theorem \ref{mt1.0}, we have
\[
\||x|^{-\lambda}\psi_\eta u\|_{L^{p'}(B_\rho)}\le \kappa \||x|^{-\lambda}\Delta (\psi_\eta u)\|_{L^{p}(B_{\rho})}
+\kappa \rho^{-\lambda}\|\Delta v\|_{L^{p}(B)},
\]
where $\kappa=\kappa(n)$ is the constant in \eqref{1}.

Similarly to \eqref{7}, we prove
\begin{align*}
\||x|^{-\lambda}\Delta (\psi_\eta u)\|_{L^{p}(B_\rho)}\le  &\||x|^{-\lambda}\psi_\eta \Delta u\|_{L^p(B_\rho)} 
\\
&c_0\eta ^{-(\lambda+2)}\left(\|\nabla u\|_{L^2(B_{2\eta})}+\|u\|_{L^2(B_{2\eta})}\right),
\end{align*}
for some constant $c_0=c_0(n)>0$. 

This inequality in the preceding one yields
\begin{align*}
\||x|^{-\lambda}\psi_\eta u\|_{L^{p'}(B_\rho)}\le \kappa  \||x|^{-\lambda}\psi_\eta &\Delta u\|_{L^p(B_\rho)}+\kappa \rho^{-\lambda}\|\Delta v\|_{L^{p}(B)}
\\
&+\kappa c_0\eta ^{-(\lambda+2)}\left(\|\nabla u\|_{L^2(B_{2\eta})}+\|u\|_{L^2(B_{2\eta})}\right).
\end{align*}
Let $P\in L^n(B,\mathbb{C}^n)$. Then, under the assumption
\[
|\Delta u|\le |qu|+|P\cdot \nabla u|\quad \mbox{in}\; B,
\]
the inequality above gives
\begin{align*}
&\||x|^{-\lambda}\psi_\eta u\|_{L^{p'}(B_\rho)}\le \kappa\left(  \||x|^{-\lambda}\psi_\eta q u\|_{L^p(B_\rho)}+\||x|^{-\lambda}\psi_\eta P\cdot\nabla u\|_{L^p(B_\rho)}\right)
\\
&\hskip 2cm +\kappa \rho^{-\lambda}\|\Delta v\|_{L^{p}(B)}+\kappa c_0\eta ^{-(\lambda+2)}\left(\|\nabla u\|_{L^2(B_{2\eta})}+\|u\|_{L^2(B_{2\eta})}\right).
\end{align*}
Since
\begin{align*}
&\||x|^{-\lambda}\psi_\eta q u\|_{L^p(B_\rho)}\le \|q\|_{L^{\frac{n}{2}}(B_\rho)}\||x|^{-\lambda}\psi_\eta  u\|_{L^{p'}(B_\rho)}
\\
&\||x|^{-\lambda}\psi_\eta P\cdot \nabla u\|_{L^p(B_\rho)}\le \|P\|_{L^n(B_\rho,\mathbb{C}^n)}\||x|^{-\lambda}\psi_\eta \nabla u\|_{L^2(B_\rho)},
\end{align*}
if $q\in \underline{Q}$ and $\|P\|_{L^n(B_{\underline{\rho}},\mathbb{C}^n)}\le \mathbf{k}$, for some constant $\mathbf{k}>0$,  we obtain
\begin{align*}
&\||x|^{-\lambda}\psi_\eta u\|_{L^{p'}(B_{\underline{\rho}})}\le 2\kappa (c_0+\mathbf{k})\eta ^{-(\lambda+2)}\|\nabla u\|_{L^2(B_{2\eta})}
\\
&\hskip 3.5cm +2\kappa \rho^{-\lambda}\|\Delta v\|_{L^{p}(B)}+2\kappa c_0\eta ^{-(\lambda+2)}\|u\|_{L^2(B_{2\eta})}.
\end{align*}

Armed with this inequality, we can proceed as in the proof of Theorem \ref{mt1.0} to obtain the following result.

\begin{theorem}\label{mt1.2}
Let $\mathbf{k}>0$, $m\in \mathbb{N}$, $\lambda \in \Lambda$, $q\in \underline{Q}$, $P\in L^n(B,\mathbb{C}^n)$ such that $\|P\|_{L^n(B_{\underline{\rho}},\mathbb{C}^n)}\le \mathbf{k}$. For all  $u\in W^{2,p}(B)\cap O_m$ satisfying $|\Delta u|\le |qu|+|P\cdot \nabla u|$ in $B$ and $0<\eta <  \frac{\underline{\rho}}{4}$, we have 
\[
\mathbf{c}\|u\|_{L^2(B_{\frac{\underline{\rho}}{2}})}\le   2^{-(\lambda+1)}\|\Delta(\phi u)\|_{L^{p}(B)}
 +3^m\eta^{m-(\lambda+3)}\varrho_m(u),
 \]
where $\mathbf{c}=\mathbf{c}(n,\mathbf{k},q_0,\underline{\rho})>0$ is a constant.
\end{theorem}

We observe that the inequality $|\Delta u|\le |qu|+|P\cdot \nabla u|$ in $B$, from Theorem \ref{mt1.2}, can be replaced by the following:
\[
c|\Delta u|\le |x|^{-2+\alpha_0}|u|+|x|^{-1+\alpha_1}|\nabla u|\quad \mbox{in}\; B,
\]
where $c>0$, $\alpha_j>0$, $j=0,1$, are constants. This follows easily by observing that $|x|^{-2+\alpha_0}\in L^{\frac{n}{2}}(B)$ and $|x|^{-1+\alpha_1}\in L^n(B)$.

The corollary below follows from Theorem \ref{mt1.2}.

\begin{corollary}\label{corollaryc2}
Let $q\in L^{\frac{n}{2}}(B)$ and $P\in L^n(B,\mathbb{C}^n)$. If $u\in W^{2,p}(B)$ is such that $|\Delta u|\le |qu|+|P\cdot \nabla u|$ in $B$ and $u$ vanishes of infinite order at $0$, then $u=0$ in a neighborhood $V$ of $0$.
\end{corollary}

Corollary \ref{corollaryc2} with $P=0$ corresponds to \cite[Theorem 6.3]{JK}. We also note that Corollary \ref{corollaryc2} does not hold in general when $n=2$ (see \cite[Theorem 2]{Wo1}).

\subsection{Quantitative SUCP for an arbitrary bounded domain}

Let $\Omega$ be a bounded domain of $\mathbb{R}^n$, $x_0 \in \Omega$, and $d := \mathrm{dist}(x_0, \mathbb{R}^n \setminus \Omega)$. Let $B_r(x_0)$ denote the open ball centered at $x_0$ with radius $r > 0$, and set
\begin{align*}
&\tilde{Q}:=\left\{\tilde{q}\in L^{\frac{n}{2}}(\Omega);\; \|\tilde{q}\|_{L^2(B_d(x_0))}\le \varsigma\right\},
\\
&\tilde{\varrho}_m(\tilde{u}):=\sup \left\{ \eta^{-m}\|\tilde{u}\|_{L^2(B_\eta (x_0))};\; 0<\eta <d\right\},\quad \tilde{u}\in L^2(\Omega),\; m\in \mathbb{N},
\\
& \tilde{O}_m:=\{\tilde{u}\in L^2(\Omega);\; \tilde{\varrho}_m(\tilde{u})<\infty\}.
\end{align*}
For $m\in \mathbb{N}$, $\tilde{u}\in W^{2,p}(\Omega)\cap  \tilde{O}_m$ and $\tilde{q}\in \tilde{Q}$, define
\[
u(x)=\tilde{u}(x_0+dx),\quad q(x)=d^2\tilde{q}(x_0+dx),\quad x\in B.
\]
We verify that $u\in W^{2,p}(B)\cap O_m$, $q\in Q$,  with
\[
\|q\|_{L^{\frac{n}{2}}(B)}=\|\tilde{q}\|_{L^{\frac{n}{2}}(B_d(x_0))},\quad \varrho_m(u)=d^{m-n/2}\tilde{\varrho}_m(\tilde{u})
\]
and
\[
(\Delta +q)u=d^2(\Delta +\tilde{q}(x_0+d\, \cdot ))\tilde{u}(x_0+d\, \cdot )\quad \mbox{in}\; B.
\]
Let $0<\eta <\frac{\rho}{4}$. Applying Theorem \ref{mt1.0},  we get
\begin{align*}
&\mathbf{c}\|u\|_{L^2(B_{\frac{\rho}{2}})}\le   \eta^{-\lambda}\|(\Delta+q) u\|_{L^p(B)}
\\
&\hskip 3cm + 2^{-(\lambda+1)}\|(\Delta+q) (\phi u)\|_{L^{p}(B)} +3^m\eta^{m-(\lambda+3)}\varrho_m(u).
\end{align*}
Here and henceforth, $\mathbf{c}=\mathbf{c}(n,r_0,\vartheta)>0$ denotes a generic constant. Therefore, we obtain
\begin{align*}
&\mathbf{c}d^{-\frac{n}{2}}\|\tilde{u}\|_{L^2(B_{\frac{d\rho}{2}})}\le   d^{2-\frac{n}{p}}\eta^{-\lambda}\|(\Delta+q) \tilde{u}\|_{L^p(B_d(x_0))}
\\
&\hskip 1.5cm + 2^{-(\lambda+1)}d^{2-\frac{n}{p}}\|(\Delta+\tilde{q}) (\tilde{\phi} \tilde{u})\|_{L^{p}(B_d(x_0))} +3^md^{m-\frac{n}{2}}\eta^{m-(\lambda+3)}\tilde{\varrho}_m(\tilde{u}),
\end{align*}
where $\tilde{\phi}(y)=\phi \left(\frac{y-x_0}{d}\right)$, $y\in B_d(x_0)$. In summary, we have  proved the following theorem.

\begin{theorem}\label{Mt1}
There exists a constant $\mathbf{c}=\mathbf{c}(n,r_0,\vartheta,d)>0$ such that for all $m\in \mathbb{N}$, $\lambda \in \Lambda$, $\tilde{q}\in \tilde{Q}$, $\tilde{u}\in W^{2,p}(\Omega)\cap \tilde{O}_m$ and $0<\eta <  \frac{\rho}{4}$, we have 
\begin{align*}
&\mathbf{c}\|\tilde{u}\|_{L^2(B_{\frac{d\rho}{2}}(x_0))}\le   \eta^{-\lambda}\|(\Delta+\tilde{q}) u\|_{L^p(B_d(x_0))}
\\
&\hskip 3cm + 2^{-(\lambda+1)}\|(\Delta+\tilde{q}) (\tilde{\phi} \tilde{u})\|_{L^{p}(B_d(x_0))}+(3d)^m\eta^{m-(\lambda+3)}\tilde{\varrho}_m(\tilde{u}).
\end{align*}
\end{theorem}

Theorem \ref{Mt1} together with \cite[Theorem 1.2 and comments in Section 1.3]{Ch} yield the following global quantitative SUCP, where 
\[
\mathbf{\tilde{Q}}:=\left\{\tilde{q}\in L^{n}(\Omega);\; \|\tilde{q}\|_{L^n(B_d(x_0))}\le |B_d(x_0)|^{-\frac{1}{n}}\varsigma \right\}\; (\subset \tilde{Q}).
\]

\begin{theorem}\label{Mt2}
Assume that $\Omega$ is $C^{1,1}$. Let $0<s,t<\frac{1}{2}$. For all $0<r <1$, $0<\eta<\frac{\rho}{4}$, $m\in \mathbb{N}$, $\lambda \in \Lambda$, $\tilde{q}\in \mathbf{\tilde{Q}}$ and $u\in H^2(\Omega)\cap \tilde{O}_m$ we have
\begin{align*}
&\mathbf{c}\|\tilde{u}\|_{H^{\frac{3}{2}+t}(\Omega)}\le r^{\zeta}\|\tilde{u}\|_{H^2(\Omega)} 
\\
&\hskip 1.5cm +e^{\mathfrak{b}r^{-\nu}}\Big(\eta^{-\lambda}\|(\Delta+\tilde{q}) \tilde{u}\|_{L^2(\Omega)} 
\\
&\hskip 4cm + 2^{-(\lambda+1)}\|(\Delta+q) (\tilde{\phi} \tilde{u})\|_{L^{p}(\Omega)}+(3d)^m\eta^{m-(\lambda+3)}\tilde{\varrho}_m(\tilde{u})\Big),
\end{align*}
where $\mathbf{c}=\mathbf{c}(n,s, t, r_0,\vartheta,\Omega, x_0,d)>0$, $\mathfrak{b}=\mathfrak{b}(n, s, t, r_0, \vartheta,\Omega,x_0,d)>0$, $\nu=\nu(n,\Omega)>0$ are constants and $\zeta =\min \left(\frac{\frac{1}{2}-t}{1+t},\frac{s}{4}\right)$.
\end{theorem}

\appendix
\section{Proof of Lemma \ref{lemma1}} \label{appendixA}

Let $u\in C_0^\infty(\dot{B})$, $\lambda \in \Lambda$ and $q\in L^{\frac{n}{2}}(B)$ satisfying
\begin{equation}\label{2}
2\kappa \|q\|_{L^{\frac{n}{2}}(B)}\le 1.
\end{equation} 
According to \eqref{1}, we have
\begin{align*}
\||x|^{-\lambda}u\|_{L^{p'}(B)}&\le \kappa \||x|^{-\lambda}(\Delta +q) u\|_{L^p(B)} +\kappa\|q|x|^{-\lambda}u\|_{L^p(B)}
\\
&\le \kappa \||x|^{-\lambda}(\Delta +q) u\|_{L^p(B)} +\kappa\|q\|_{L^{\frac{n}{2}}(B)}\||x|^{-\lambda}u\|_{L^{p'}(B)},
\end{align*}
which, combined with \eqref{2}, gives
\begin{equation}\label{3}
\||x|^{-\lambda}u\|_{L^{p'}(B)}\le 2\kappa \||x|^{-\lambda}(\Delta +q)u\|_{L^p(B)}.
\end{equation}

Let $\mathbf{R}$ be the set consisting of the constants $r>0$ for which there exists $\kappa_r>0$ such that for all $u\in C_0^\infty(\dot{B})$, $\lambda \in \Lambda$ and $q\in \mathbf{B}_r$, we have
\begin{equation}\label{3.0}
\||x|^{-\lambda}u\|_{L^{p'}(B)}\le \kappa_r \||x|^{-\lambda}(\Delta +q)u\|_{L^p(B)}.
\end{equation}

By \eqref{3}, $(0,(2\kappa)^{-1}]\subset \mathbf{R}$. Furthermore, we verify that if $0<r'<r$ and $r\in \mathbf{R}$, then $r'\in \mathbf{R}$. That is $R$ is an interval. We now show that $R=]0,r_\ast[$, where $r_\ast=\sup R$. For this, let $r\in \mathbf{R}$ and $r_1=r+ (2\kappa_r)^{-1}$. For all $u\in C_0^\infty(\dot{B})$, $\lambda \in \Lambda$ and  $q\in \mathbf{B}_{r_1}$, using 
\[
q=\frac{r}{r_1}q+\frac{(2\kappa_r)^{-1}}{r_1}q,
\]
and proceeding as for \eqref{3},  we obtain from \eqref{3.0}
\[
\||x|^{-\lambda}u\|_{L^{p'}(B)}\le 2\kappa_r \||x|^{-\lambda}(\Delta +q)u\|_{L^p(B)}.
\]
That is $r_1\in \mathbf{R}$ and hence $\mathbf{R}$ is an open interval. The proof is complete.

%\subsection*{Conflict of Interest statement}The author declares that there is no conflict of interest regarding the content of this work.

%\subsection*{Data availability statement} There is no data associated with this work.

\end{document}